\documentclass[12pt,reqno]{amsart}
\usepackage{fullpage}
\usepackage{times}
\usepackage{amsmath,amssymb,amsthm,url}
\usepackage[utf8]{inputenc}
\usepackage[english]{babel}
\usepackage{bbm}
\usepackage{enumerate}
\usepackage{bm}
\usepackage{graphicx}
\usepackage[colorlinks=true, pdfstartview=FitH, linkcolor=blue, citecolor=blue, urlcolor=blue]{hyperref}

\newtheorem{thm}{Theorem}[section]
\newtheorem{lem}[thm]{Lemma}
\newtheorem{prop}[thm]{Proposition}
\newtheorem{defn}[thm]{Definition}

\newtheorem{cor}[thm]{Corollary}
\newtheorem{rem}[thm]{Remark}

\newtheorem{conj}[thm]{Conjecture}

\usepackage{tabstackengine}
\stackMath

\newcommand{\F}{\mathbb{F}}
\newcommand{\N}{\mathbb{N}}
\makeatletter
\@namedef{subjclassname@2020}{%
  \textup{2020} Mathematics Subject Classification}
\makeatother
\begin{document}

\title{On the clique number of Paley graphs of prime power order}
\author{Chi Hoi Yip}
\address{Department of Mathematics \\ University of British Columbia \\ 1984 Mathematics Road \\ Canada V6T 1Z2}
\email{kyleyip@math.ubc.ca}

\date{\today}

\maketitle

\begin{abstract}
Finding a reasonably good upper bound for the clique number of Paley graphs is an open problem in additive combinatorics. A recent breakthrough by Hanson and Petridis using Stepanov's method gives an improved upper bound on Paley graphs defined on a prime field $\mathbb{F}_p$, where $p \equiv 1 \pmod 4$. We extend their idea to the finite field $\mathbb{F}_q$, where $q=p^{2s+1}$ for a prime $p\equiv 1 \pmod 4$ and a non-negative integer $s$. We show the clique number of the Paley graph over $\mathbb{F}_{p^{2s+1}}$ is at most $\min \bigg(p^s \bigg\lceil \sqrt{\frac{p}{2}} \bigg\rceil, \sqrt{\frac{q}{2}}+\frac{p^s+1}{4}+\frac{\sqrt{2p}}{32}p^{s-1}\bigg)$. 
\end{abstract}

\section{Introduction}

Throughout the paper, we let $p$ be an odd prime, $r$ a positive integer such that $q=p^r \equiv 1 \pmod{4}$. Let $\F_q$ be the finite field with $q$ elements, and $\F_q^*=\F_q \setminus \{0\}$. For an undirected graph $G$, the clique number of $G$, denoted $\omega (G)$, is the size of a maximum clique of $G$. 

Let $q \equiv 1 \pmod 4$ be a prime power. The {\em Paley graph} defined on the finite field $\F_q$, denoted $P_q$, is the graph whose vertices are elements in $\F_q$ such that two vertices are adjacent if and only if their difference is a quadratic residue in $\F_q^*$. The condition $q \equiv 1 \pmod{4}$ is needed so that $P_q$ is undirected. 

We are interested in finding an upper bound for the size of a maximum clique $C$ of the Paley graph $P_q$. Note that for any $x\in \F_q$, $C-x$ also gives a maximum clique. Throughout the paper, let
\begin{equation}\label{maxclique}
C=\{a_1,a_2,\ldots,a_N\}    
\end{equation}
be a maximum clique in $P_q$, i.e., $N=\omega(P_q)$.  Without loss of generality, we may assume $a_1=0$. It follows that $a_2,\ldots,a_N$ are quadratic residues in $\F_q^*$ and we obtain the following naive upper bound on $\omega(P_q)$, since the number of quadratic residues in $\F_q^*$ is $\frac{q-1}{2}$.

\begin{lem}\label{t0}
If $q=p^r \equiv 1 \pmod{4}$, then $\omega(P_q) \leq \frac{q+1}{2}$.
\end{lem}

Paley graphs have many nice properties; see for example the survey \cite[Chapters 2]{thesis}. In particular, since $P_q$ is self-complementary, it is easy to improve the upper bound on $\omega(P_q)$ to $\sqrt{q}$.  This square root upper bound is known as {\em the trivial upper bound} on the clique number of a Paley graph; see for example the literature~\cite{BDR, BMR, SC, CL}. For the sake of completeness, we outline a simple proof in the following lemma.

\begin{lem}\label{tt}
If $q=p^r \equiv 1 \pmod{4}$, then $\omega(P_q) \leq \sqrt{q}$.
\end{lem}
\begin{proof}
Let $C$ be a maximum clique defined in equation \eqref{maxclique}. Let $r$ be a quadratic non-residue in $\F_q^*$, we consider the set $A=\{a_i+ra_j: 1 \leq i,j \leq N\}$. Note that if $a_i+ra_j=a_i'+ra_j'$, then $a_i-a_i'=r(a_j'-a_j)$. If $i \neq i'$ or $j \neq j'$, then we will have a quadratic residue equals a quadratic non-residue, which is impossible. So each element of $A$ is different from the others. This means that $|A|=N^2 \leq q$, i.e. $N \leq \sqrt{q}$.
\end{proof}

When $q=p^{2s}$, the upper bound $\sqrt{q}$ can be achieved by considering the subfield $\F_{p^s}$ as a Paley clique \cite{BDR}. So there is no way to improve the upper bound for the case $q$ is an even power of $p$. Therefore, our focus will be on the case when $q$ is an odd power of $p$. It is widely believed $\omega(P_p) \ll_{\epsilon} p^{\epsilon}$ for any $\epsilon>0$, which is consistent with numerical evidence. Computer experiments \cite{BR,EX} suggest that the correct order of $\omega(P_q)$ should be a polylogarithmic function in $q$, where $q$ is an odd power of a prime $p \equiv 1 \pmod 4$.

Next we revisit the known lower bounds on the clique number. 
Cohen \cite{SC} showed that $\omega(P_q) \gg \log q$. Graham and Ringrose \cite{GR} showed that the least positive integer $n(p)$ that is a quadratic non-residue modulo $p$ is of the size $\Omega(\log p \log \log \log p)$. Moreover,  Montgomery \cite{HLM} showed that this can be improved to $\Omega(\log p \log \log p)$ under the generalized Riemann hypothesis (GRH). Note that for each prime $p \equiv 1 \pmod 4$, the set $\{0,1, \cdots, n(p)-1\}$ forms a Paley clique due to the definition of $n(p)$, thus $\omega(P_p)=\Omega(\log p \log \log \log p)$. And if GRH is true, then $\omega(P_p)=\Omega(\log p \log \log p)$. These results on the lower bound of the clique number are consistent with the computer experiments.

However, finding a reasonably good upper bound remains to be an open problem in additive combinatorics \cite{CL}. The current best upper bound for $\omega(P_q)$ is of the order $\sqrt{q}$, which is the same as the above trivial bound given in Lemma \ref{tt}. Therefore, there is still a huge gap between the optimal upper bound we currently have and the lower bound for the clique number. For the case $q=p$, the current best-known bound is $\sqrt{\frac{p}{2}}+1$, which was proved by Hanson and Petridis \cite{HP} using Stepanov's method.

\begin{thm}[\cite{HP}] \label{HP}
If $p \equiv 1 \pmod{4}$, then $\omega(P_p) \leq \frac{\sqrt{2p-1}+1}{2}$.
\end{thm}

Recently, Di Benedetto,  Solymosi, and White \cite{DSW} used R\'edei polynomial with Sz\H{o}nyi's extension to derive improved lower bound on the number of directions determined by a Cartesian product.

\begin{thm}[\cite{DSW}] \label{DSW}
Let $A,B\subset \F_p$ be sets each of size at least two such that $|A||B| < p$. Then the set of points $A\times B\subset \F_p^2$ determines at least $|A||B| - \min\{|A|,|B|\} + 2$ directions.
\end{thm}

As a corollary, if we take $A$ to be a clique in a Paley graph over $\F_p$, then each direction determined by $A \times A$ is either a quadratic residue in $\F_p^*$, 0, or $\infty$. So the number of directions determined by $A \times A$ is at most $\frac{p+3}{2}$, and one can recover Hanson-Petridis bound.  

Nevertheless, it is worthwhile to point out that both the polynomial method used in the proof of Theorem \ref{HP} and the key lemma \cite[Lemma 6]{DSW} used to prove Theorem \ref{DSW} only work on the prime field $\F_p$. Recently, the author \cite{Yip2} generalized Theorem \ref{DSW} into $\F_q^2$; however, the extension \cite[Theorem 6]{Yip2} has several technical assumptions due to the subfield obstruction, which prevents us from improving the trivial upper bound on $\omega(P_q)$ in a similar manner.

In this paper, we extend the idea of Hanson and Petridis \cite{HP} and give an improvement on the upper bound of $\omega(P_q)$. Before stating our main result, we recall the best known upper bound on $\omega(P_q)$, due to Bachoc, Matolcsi, and Ruzsa \cite{BMR}.

\begin{thm}[\cite{BMR}] \label{BMR}
Assume $p \equiv 1 \pmod{4}$ and $q=p^{2s+1}$. Let $\omega=\omega(P_q)$ be the clique number of $P_q$. Then 
\begin{itemize}
    \item If $[\sqrt{q}]$ is even then $\omega^2+\omega-1\leq q$.
    \item If $[\sqrt{q}]$ is odd then $\omega^2+2\omega-2\leq q$.
\end{itemize}
\end{thm}
Theorem \ref{BMR} implies that $\omega(P_q)\leq \sqrt{q}-1$ for approximately three quarters prime powers $q$ of the form $p^{2s+1}$.   Greaves and Soicher \cite{GS} provided a different proof of Theorem \ref{BMR} using algebraic graph theory; essentially they improved the ratio bound on the clique number.

The following theorem is the key to derive an improved upper bound on $\omega(P_q)$.

\begin{thm} \label{t1}
If $q=p^{2s+1}, p\equiv 1 \pmod{4}$, and $2 \leq n\leq N=\omega(P_q)$ satisfies 
$
\binom{n-1+\frac{q-1}{2}}{\frac{q-1}{2}}\not \equiv 0 \pmod p,
$
then $(N-1)n \leq \frac{q-1}{2}$.
\end{thm}

In view of the statement of Theorem \ref{t1}, it is crucial to determine whether a binomial coefficient is divisible by a prime $p$. One tool that is useful for this purpose is Lucas's theorem, which states that 
if $p$ is a prime and if $m,n$ are non-negative integers with base-$p$ representation
$$
m=m_{k}p^{k}+m_{k-1}p^{k-1}+\cdots +m_{1}p+m_{0}=(m_k, m_{k-1}, \cdots, m_{0})_p,
$$
$$
n=n_{k}p^{k}+n_{k-1}p^{k-1}+\cdots +n_{1}p+n_{0}=(n_k, n_{k-1}, \cdots, n_{0})_p,
$$
where $0 \leq m_j,n_j \leq p-1$ for each $0 \leq j \leq k$, then 
$$
\binom{m}{n} \equiv \prod_{j=0}^{k} \binom{m_j}{n_j} \pmod p.
$$
In particular, $\binom{m}{n} \not \equiv 0 \pmod p$ if and only if $n_j \leq m_j$ for each $0 \leq j \leq k$.

The following theorem is our main result, which can be deduced from Theorem \ref{t1} and Lucas's theorem.

\begin{thm}\label{main}
Assume $p \equiv 1 \pmod{4}$ and $q=p^{2s+1}$ for some nonneagtive integer $s$, then 
\begin{equation} \label{maint}
\omega(P_q) \leq \min \bigg(p^s \bigg\lceil \sqrt{\frac{p}{2}} \bigg\rceil, \sqrt{\frac{q}{2}}+\frac{p^s+1}{4}+\frac{\sqrt{2p}}{32}p^{s-1}\bigg).    
\end{equation}
\end{thm}

In particular, Theorem \ref{main} implies that $\omega(P_q) \leq \sqrt{q/2}+o(\sqrt{q})$, which improves Theorem \ref{BMR} by a multiplicative constant. One can use the uniform distribution of the fractional parts of $\{\sqrt{p}: p \equiv 1 \pmod 4\}$ (see for example \cite[Corollary 6.4]{Yip2}) to deduce that both bounds on the right-hand-side of equation \eqref{maint} take the lead infinitely often. 

We will extend the notion of derivatives to the finite field in Section \ref{sec2}. The proof of the main result will be given in Section \ref{sec3}. In Section \ref{sec4}, we will describe a variant of Theorem \ref{t1}, which possibly leads to the following (further) improved upper bound on $\omega(P_q)$.

\begin{conj}\label{conj1}
There is some constant $c>0$, such that if $p \equiv 1 \pmod{4}$, and $q=p^{2s+1}$ for some positive integer $s$, then $\omega(P_q) \leq \sqrt{\frac{q}{2}}+cp^{s-1}$.
\end{conj}

\section{Hyper-derivatives}\label{sec2}

The following is a well-known relation between the multiplicity of roots and the derivatives.

\begin{lem}\label{rootR}
Let $0 \neq f \in K[x]$, where $K$ is a field with characteristic zero. Suppose $c$ is a root of $f^{(n)}(x)$ for $n=0,1,\ldots, m-1$, then $c$ is a root of multiplicity at least $m$. 
\end{lem}

However, the same result fails to hold for fields with nonzero characteristic. This is because if $\operatorname{char} K=p>0$, then for any polynomial $f \in K[x]$, we have $f^{(p)} (x) \equiv 0$. This means we need to modify the definition of derivative in order to overcome the nonzero characteristic, and a good idea is to introduce the binomial coefficients into the derivatives \cite{LN}.

\begin{defn}
Let $K$ be a field and let $b_0,b_1, \ldots b_d \in K$. If $n$ is a non-negative integer, then the $n$-th order hyper-derivative of $f(x)=\sum_{j=0}^d b_j x^j$ is
$$
E^{(n)}(f) =\sum_{j=0}^d \binom{j}{n} b_j x^{j-n}.
$$
\end{defn}

Hyper-derivatives are also known as Hasse derivatives. 
Note that $E^{(1)}$ matches with the usual first order derivative. And if $\operatorname{char} K=0,$ or $\operatorname{char} K>n!$ then $E^{(n)}(f)=\frac{1}{n!} f^{(n)}$. Readers can refer to \cite[Chapter 6]{LN} for a comprehensive applications of hyper-derivatives. Here we provide a self-contained overview of some simple properties of hyper-derivatives. 

The following is analogous to the Leibniz rule for standard derivatives.

\begin{lem} [Leibniz rule for hyper-derivatives]
If $f_1, \ldots, f_t \in K[x]$, then
$$
E^{(n)}(f_1 \ldots f_t)= \sum_{\substack{n_1,\ldots n_t \geq 0,\\ n_1+ \ldots +n_t=n}} E^{(n_1)} (f_1) \ldots E^{(n_t)} (f_t) 
$$
\end{lem}
\begin{proof}
Note that hyper-derivatives are linear. So it suffices to consider the case for monomials. Assume $f_j(x)=x^{k_j}$ for $1 \leq j \leq t$ and $k=\sum_{j=1}^t k_j$. Then
$$
E^{(n)}(f_1 \ldots f_t)=E^{(n)}(x^k)=\binom{k}{n} x^{k-n},
$$
$$
\sum_{\substack{n_1,\ldots n_t \geq 0,\\ n_1+ \ldots +n_t=n}} E^{(n_1)} (f_1) \ldots E^{(n_t)} (f_t)
=\sum_{\substack{n_1,\ldots n_t \geq 0,\\ n_1+ \ldots +n_t=n}} \Bigg(\prod_{j=1}^t \binom{k_j}{n_j}\Bigg) x^{k-n}.
$$
Consider the coefficient of $x^n$ of the two sides of the identity $(1+x)^k=\prod_{j=1}^{t} (1+x)^{k_j}$, we get
$$
\binom{k}{n}=\sum_{\substack{n_1,\ldots n_t \geq 0,\\ n_1+ \ldots +n_t=n}} \prod_{j=1}^t \binom{k_j}{n_j},
$$
which proves the proposition.
\end{proof}

\begin{cor}
$E^{(n)}\big((x-c)^t\big)=\binom{t}{n} (x-c)^{t-n}.$
\end{cor}

\begin{proof}
For $1 \leq i \leq t$, let $f_i(x)=x-c$, then $E^{(1)} (f_i)=1$, and $E^{(k)} (f_i)=0$ for $k \geq 2$.
So by the Leibniz rule, we have
$$
E^{(n)}\big((x-c)^t\big)= \sum_{\substack{n_1,\ldots n_t \in \{0,1\},\\ n_1+ \ldots +n_t=n}} (x-c)^{t-n}=\binom{t}{n} (x-c)^{t-n}.
$$
\end{proof}
Now we can establish a relation between the multiplicity of roots and the hyper-derivatives parallel to Lemma \ref{rootR}, which is crucial for the proof of our main results.

\begin{lem}\label{root}
Let $0 \neq f \in K[x]$. Suppose $c$ is a root of $E^{(n)}(f)$ for $n=0,1,\ldots, m-1$, then $c$ is a root of multiplicity at least $m$. 
\end{lem}
\begin{proof}
Let $f(x)=b_0+b_1(x-c)+\cdots +b_d (x-c)^d$, where $b_d \neq 0$, then
$$
E^{(n)}(f)(x)=b_n+ \binom{n+1}{n} b_{n+1} (x-c)+ \ldots +\binom{d}{n} b_d (x-c)^{d-n}.
$$
Since $c$ is a root of $E^{(n)}(f)$ for $n=0,1,\ldots, m-1$, then $b_n=0$ for $n=0,1,\ldots, m-1$. So $c$ is a root of multiplicity at least $m$. 
\end{proof}

\section{Proof of the Main Result}\label{sec3}
In this section, we first prove Theorem \ref{t1} and then use it to deduce Theorem \ref{main}. Recall that $C=\{a_1,a_2, \cdots, a_N\}$, defined in equation \eqref{maxclique}, is a maximum clique in $P_q$.

\begin{lem} \label{lem1}
For each $1 \leq i,j \leq N$, $i \neq j$, we have $(a_i-a_j)^{\frac{q-1}{2}}=1$.
\end{lem}

\begin{proof}
Let $x=a_i-a_j$, then $x \in \F_q^*$ and $x$ is a quadratic residue in $\F_q$, so $x=y^2$ for some $y \in \F_q^*$. Since $y^{q-1}=1$, we have $x^{\frac{q-1}{2}}=1$. 
\end{proof}

\begin{cor}\label{cor}
For any $k \in \N$, and any $1 \leq i,j \leq N$, we have
${(a_i-a_j)^{\frac{q-1}{2}+k}=(a_i-a_j)^{k}.}$
\end{cor}
\begin{proof}
By Lemma \ref{lem1}, if $1 \leq i,j \leq N$, and $i \neq j$, then $(a_i-a_j)^{\frac{q-1}{2}}=1$, so for any $k>0, (a_i-a_j)^{\frac{q-1}{2}+k}=(a_i-a_j)^{k}$. Note the previous equation also holds when $i=j$ since both sides are zero. 
\end{proof}

Now we are ready to prove Theorem \ref{t1}, which can be regarded as a generalization of Theorem \ref{HP}. Note that the original proof of Theorem \ref{HP} by Hanson and Petridis implicitly uses the fact that when $q$ is a prime, the binomial coefficient $$\binom{\omega(P_q)-1+\frac{q-1}{2}}{\frac{q-1}{2}}\not \equiv 0 \pmod p.$$ This condition is crucial to ensure the polynomial they constructed is nonzero. However, this condition no longer holds if we are working on a finite field $\F_q$, and we shall see the main difficulty in extending their method to $\F_q$ is to optimize an upper bound while ensuring the polynomial we are interested in is not identically zero.

\begin{proof}[Proof of Theorem \ref{t1}]
Consider the following polynomial 
$$
f(x)=\sum_{i=1}^n c_i (x-a_i)^{n-1+\frac{q-1}{2}} -1 \in \F_q[x],
$$
where $c_1,c_2,...,c_n$ is the unique solution of the following system of equations:
\begin{equation} \label{system} \tag{F}
\left\{
\TABbinary\tabbedCenterstack[l]{
\sum_{i=1}^n c_i (-a_i)^j=0,  \quad 0 \leq j \leq n-2\\\\
\sum_{i=1}^n c_i (-a_i)^{n-1}=1
}\right.    
\end{equation}
Note the above system of equations has a unique solution since the coefficient matrix of the system is a Vandermonde matrix with parameters $a_1, a_2, \ldots a_n$ all distinct. 
For each $0 \leq k \leq n-1+\frac{q-1}{2}$,
the coefficient of $x^{n-1+\frac{q-1}{2}-k}$
is 
$$
\sum_{i=1}^n  \binom{n-1+\frac{q-1}{2}}{k} c_i (-a_i)^{k}=\binom{n-1+\frac{q-1}{2}}{k} \sum_{i=1}^n c_i (-a_i)^{k}.
$$
Now by our construction, the coefficient of $x^{n-1+\frac{q-1}{2}-k}$
is $0$ for $k=0,1, \ldots, n-2$, and the coefficient of $x^{\frac{q-1}{2}}$
is $\binom{n-1+\frac{q-1}{2}}{n-1}= \binom{n-1+\frac{q-1}{2}}{\frac{q-1}{2}}\not \equiv 0 \pmod p$, so the degree of $f$ is $\frac{q-1}{2}$.

Note that for each $1\leq j \leq N$, by Corollary \ref{cor}, we have
\begin{align*}
E^{(0)} f (a_j)
&= f (a_j)\\
&=  \sum_{i=1}^n c_i (a_j-a_i)^{n-1+\frac{q-1}{2}} -1 \\
&=  \sum_{i=1}^n c_i (a_j-a_i)^{n-1}  -1 \\
&=  \sum_{l=0}^{n-1} \binom{n-1}{l} \bigg(\sum_{i=1}^n c_i (-a_i)^l\bigg)a_j^{n-1-l}  -1  \\
&= \sum_{i=1}^n c_i (-a_i)^{n-1}  -1  \\
&=0.
\end{align*}
For each $1\leq j \leq N$ and $1 \leq k \leq n-2$, again by Corollary \ref{cor}, we have
\begin{align*}
E^{(k)} f (a_j)
&= \binom{n-1+\frac{q-1}{2}}{k}  \sum_{i=1}^n c_i (a_j-a_i)^{n-1+\frac{q-1}{2}-k} \\
&= \binom{n-1+\frac{q-1}{2}}{k}  \sum_{i=1}^n c_i (a_j-a_i)^{n-1-k}   \\
&= \binom{n-1+\frac{q-1}{2}}{k}  \sum_{l=0}^{n-1-k} \binom{n-1-k}{l} \bigg(\sum_{i=1}^n c_i (-a_i)^l\bigg)a_j^{n-1-k-l}   \\
&=0.
\end{align*}
For each $n+1\leq j \leq N$, by Lemma \ref{lem1}, we additionally have
\begin{align*}
E^{(n-1)} f (a_j)
= \binom{n-1+\frac{q-1}{2}}{n-1}  \sum_{i=1}^n c_i (a_j-a_i)^{\frac{q-1}{2}} = \binom{n-1+\frac{q-1}{2}}{n-1}  \sum_{i=1}^n c_i  =0.
\end{align*}
Now by Lemma \ref{root}, each of $a_1,a_2, \ldots a_n$ is a root of $f$ of multiplicity at least $n-1$, and each of $a_{n+1},a_{n+2}, \ldots a_N$ is a root of $f$ of multiplicity at least $n$. 
Therefore
$$
n(n-1)+(N-n)n= (N-1)n \leq \operatorname{deg}f =\frac{q-1}{2}.
$$
\end{proof}

Now we can see that Theorem \ref{HP} is a special case of Theorem \ref{t1}.

\begin{proof} [Proof of Theorem \ref{HP}]
By Lemma \ref{t0}, $N=\omega(P_p) \leq \frac{p+1}{2}$, so ${N-1+\frac{p-1}{2}<p}$ and $ \binom{N-1+\frac{p-1}{2}}{\frac{p-1}{2}} \not \equiv 0 \pmod p$. Hence by Theorem \ref{t1}, we have $N(N-1)\leq \frac{p-1}{2}$, and ${N \leq \frac{1}{2}(\sqrt{2p-1}+1)}$.
\end{proof}

\begin{cor} \label{ch}
If $p\equiv 1 \pmod{4}$, then $\omega(P_p) \leq 
\lceil \sqrt{\frac{p}{2}} \rceil$.
\end{cor}
\begin{proof}
Assume $p=2r^2+t$, where $1 \leq t \leq 4r+1$, then $2r^2<p<2(r+1)^2$, $\lceil \sqrt{\frac{p}{2}} \rceil=r+1$. Note $$2p-1=4r^2+(2t-1) \leq 4r^2+8r+1 <(2r+2)^2,$$ so $\sqrt{2p-1}<2r+2$. Then by Theorem \ref{HP}, $\omega(P_p)\leq  \frac{1}{2}(\sqrt{2p-1}+1) \leq r+1+\frac{1}{2}$. Since $\omega(P_p)$ is an integer, we have $\omega(P_p) \leq r+1=\lceil \sqrt{\frac{p}{2}} \rceil$.
\end{proof}

Next, we use Theorem \ref{t1} to deduce Theorem \ref{main}. To do so, we need to provide an ad-hoc analysis on the binomial coefficients in each separate case.

\begin{prop}\label{t2}
If $ p\equiv 1 \pmod{4}$, then for $q=p^{3}$, $N=\omega(P_q)$ satisfies 
$(N-1)(N-\frac{p-1}{2}) \leq \frac{q-1}{2}$.
\end{prop}

\begin{proof}
Without loss of generality, we may assume $\frac{1}{2}(\sqrt{2q-1}+1) < N < \sqrt{q}$. Suppose the base-$p$ representation of $N-1$ is $N-1=(A,B)_p$, then $\lfloor \sqrt{\frac{p-1}{2}}\rfloor \leq A \leq \lfloor \sqrt{p}\rfloor.$  Let $n-1=(A, b)_p$, where $b=\min(\frac{p-1}{2},B)$, then $n \leq N\leq n+\frac{p-1}{2}$. Then 
$$n-1+\frac{q-1}{2}=(\frac{p-1}{2},A+\frac{p-1}{2}, b+\frac{p-1}{2})_p,$$
and by Lucas's theorem, 
$$
\binom{n-1+\frac{q-1}{2}}{\frac{q-1}{2}} =\binom{\frac{p-1}{2}}{\frac{p-1}{2}} \binom{A+\frac{p-1}{2}}{\frac{p-1}{2}} \binom{b+\frac{p-1}{2}}{\frac{p-1}{2}} \not \equiv 0 \pmod p.
$$
So by Theorem \ref{t1}, $(N-1)(N-\frac{p-1}{2}) \leq (N-1)n \leq \frac{q-1}{2}$.
\end{proof}

\begin{prop}\label{t3}
If $q=p^{2s+1},p\equiv 1 \pmod{4}$, and $s \geq 2$, then $N=\omega(P_q)$ satisfies
$(N-1)(N-\frac{p^s-1}{2}) \leq \frac{q-1}{2}$.
\end{prop}

\begin{proof}
Without loss of generality, we may assume $\frac{1}{2}(\sqrt{2q-1}+1) < N < \sqrt{q}$. Suppose the base-$p$ representation of $N-1$ is 
$$
N-1=(z_s,z_{s-1},...,z_0)_p,
$$
then since $p \geq 5$, we have
$
1 \leq \lfloor \sqrt{\frac{p-1}{2}}\rfloor \leq z_s \leq \lfloor \sqrt{p}\rfloor \leq \frac{p-1}{2}.
$
\begin{itemize}
\item If $z_{s-1}\leq \frac{p-1}{2}$, let 
$$
n-1=(z_s, z_{s-1}, z_{s-2}', \cdots, z_0')_p,
$$
to be the largest number no greater than $N$ such that $z_j' \leq \frac{p-1}{2}$ for each ${0 \leq j \leq s-2}$. Then $n \leq N \leq n+\frac{1}{2} (p^{s-1}-1)$, and Lucas's theorem implies that
$$
\binom{n-1+\frac{q-1}{2}}{\frac{q-1}{2}} \equiv \binom{z_s+{\frac{p-1}{2}}}{{\frac{p-1}{2}}}\binom{z_{s-1}+\frac{p-1}{2}}{\frac{p-1}{2}} \prod_{j=0}^{s-2} \binom{\frac{p-1}{2}+z_j'}{\frac{p-1}{2}}\not \equiv 0 \pmod p.
$$
\item If $z_{s-1}> \frac{p-1}{2}$, let 
$$
n-1=(z_s, \frac{p-1}{2}, \ldots, \frac{p-1}{2})_p,
$$
Then 
$$
n \leq N \leq n+ p^{s}-1 -\frac{p^s-1}{2}=n+\frac{p^s-1}{2}, 
$$
and Lucas's theorem implies that
$$
\binom{n-1+\frac{q-1}{2}}{\frac{q-1}{2}} \equiv \binom{z_s+{\frac{p-1}{2}}}{{\frac{p-1}{2}}} 
\binom{p-1}{\frac{p-1}{2}}^{s}\not \equiv 0 \pmod p.
$$
\end{itemize}
In both cases, we have $n \leq N \leq n+\frac{p^s-1}{2}$, and $\binom{n-1+\frac{q-1}{2}}{\frac{q-1}{2}}\not \equiv 0 \pmod p$. So by Theorem \ref{t1}, we have 
$(N-1)(N-\frac{p^s-1}{2})\leq (N-1)n \leq \frac{q-1}{2}$.
\end{proof}

\begin{thm}\label{t4}
If $q=p^{2s+1},p\equiv 1 \pmod{4}$, and $s$ is a nonnegative integer, then 
$$
\omega(P_q)< \sqrt{\frac{q}{2}}+\frac{p^s+1}{4}+\frac{\sqrt{2p}}{32}p^{s-1}.
$$
\end{thm}

\begin{proof}
When $s=0$, by Theorem \ref{HP}, we have
$$
\omega(P_p) \leq \frac{\sqrt{2p-1}+1}{2} <\frac{\sqrt{2p}+1}{2}=\sqrt{\frac{p}{2}}+\frac{1}{2} <\sqrt{\frac{p}{2}}+\frac{p^0+1}{4}+\frac{\sqrt{2p}}{32}p^{-1}.
$$
When $s \geq 1$, by Proposition \ref{t2} and Proposition \ref{t3}, we have 
$$\omega(P_q)^2-\frac{p^s+1}{2}\omega(P_q) \leq \frac{q-p^s}{2},$$
and therefore 
\begin{align*}
\omega(P_q)
&\leq \frac{p^s+1}{4}+\sqrt{\frac{q-p^s}{2}+(\frac{p^s+1}{4})^2}\\
&=\frac{p^s+1}{4}+\sqrt{\frac{q}{2}+\frac{p^{2s}-6p^s+1}{16}} \\
&<\frac{p^s+1}{4}+\sqrt{\frac{q}{2}+\frac{p^{2s}}{16}}\\
%&<\frac{p^s+1}{4}+\sqrt{\frac{q}{2}}+\frac{p^{2s}}{16\sqrt{2q}}\\
&<\frac{p^s+1}{4}+\sqrt{\frac{q}{2}}+\frac{\sqrt{2p}}{32}p^{s-1}.
%&=p^s (\sqrt{\frac{p}{2}}+\frac{1}{4}+\frac{1}{16\sqrt{p}}+\frac{1}{4p^s}).
\end{align*}
\end{proof}

\begin{cor}\label{co2}
If $q=p^{2s+1}, p\equiv 1 \pmod{4}$, and $s \geq 1$, then 
$
\omega(P_q) \leq \lceil \sqrt{\frac{p}{2}} \rceil p^s.
$
\end{cor}

\begin{proof}
Suppose $\omega(P_q) > \lceil \sqrt{\frac{p}{2}} \rceil p^s.$ Then by Theorem \ref{t4}, 
$$
\bigg\lceil \sqrt{\frac{p}{2}} \bigg\rceil p^s +1 \leq \omega(P_q)<\sqrt{\frac{q}{2}}+\frac{p^s+1}{4}+\frac{\sqrt{2p}}{32}p^{s-1}.
$$
Then if $q=p^3$, we must have $p>5$, i.e. $p \geq 13$, since for $p=5$, we have ${\omega(P_{125})=7 <2 \cdot 5}$.
And when $p \geq 13$, we have 
$$
1 \leq \omega(P_q)-\bigg\lceil \sqrt{\frac{p}{2}} \bigg\rceil p \leq 1+\frac{p+1}{4}+\frac{\sqrt{2p}}{32}<\frac{p-1}{2}.
$$
Then as in the proof of Proposition \ref{t2}, we have $\omega(P_q)(\omega(P_q)-1) \leq \frac{q-1}{2}$, and hence
$$
\omega(P_q) \leq \frac{\sqrt{2q-1}+1}{2}<\sqrt{\frac{q}{2}}+\frac{1}{2}<\bigg\lceil \sqrt{\frac{p}{2}} \bigg \rceil p^s +1,
$$
a contradiction.\\
If $q=p^{2s+1}$ for some $s \geq 2$, then 
$$
1 \leq \omega(P_q)-\bigg\lceil \sqrt{\frac{p}{2}} \bigg\rceil p^s \leq 1+\frac{p^s+1}{4}+\frac{\sqrt{2p}}{32}p^{s-1}<\frac{p^s-1}{2}.
$$
Then as in the proof of Proposition \ref{t3}, we have $\omega(P_q)(\omega(P_q)-1) \leq \frac{q-1}{2}$, and hence
$\omega(P_q) <\lceil \sqrt{\frac{p}{2}} \rceil p^s +1,$
a contradiction.
\end{proof}

Now we are ready to prove the main result.

\begin{proof}[Proof of Theorem \ref{main}]
The first bound on the right-hand-side of equation \eqref{maint} follows from Corollaries \ref{ch} and \ref{co2}, and the second bound follows from Theorem \ref{t4}.
\end{proof}

\section{A variant of Theorem \ref{t1}}\label{sec4}

Observe that in the proof of Theorem \ref{t1}, not every equation of the system \eqref{system} is really needed. In fact, some of them will be unnecessary due to the vanishing binomial coefficients (recall we are working on a field with characteristic $p$). For $n \in \N$, let $L(n)$ denote the set of the integers $l$ such that $0 \leq l \leq n-1$ and there exists a $k$ such that $0 \leq k \leq n-1$, and 
$$ \binom{n-1+\frac{q-1}{2}}{k} \binom{n-1-k}{l} \not \equiv 0 \pmod p.$$
It turns out that only the rows with indices in the set $L(n)$ are needed, and we are able to generalize Theorem \ref{t1} by introducing a new parameter $m$.

\begin{thm} \label{t5}
Suppose $q=p^{2s+1}, p\equiv 1 \pmod{4}$, $n \leq N=\omega(P_q)$, $\frac{q-1}{2} \ge m>n$ and $\binom{n-1+\frac{q-1}{2}}{m} \not \equiv 0 \pmod p$. If $D=\{d_1, d_2, \ldots, d_n\}$ is a $n$-subset of $C$ defined in equation \eqref{maxclique}, such that the following system of equations 
\begin{equation}\label{refined_system}\tag{$E_{n,m,D}$}   
 \left\{
\TABbinary\tabbedCenterstack[l]
{
\sum_{i=1}^n c_i (-d_i)^{l}=0, \quad \forall l \in L(n)\\\\
\sum_{i=1}^n c_i (-d_i)^m =1
}\right.
\end{equation}
has a solution $c_1,c_2,...,c_n$, then $n(N-2) \leq \frac{q-3}{2}.$
\end{thm}
\begin{proof}
Consider the polynomial
$$
f(x)=\sum_{i=1}^n c_i (x-d_i)^{n-1+\frac{q-1}{2}}.
$$
Note that $f$ is a nonzero polynomial since the coefficient of $x^{n-1+\frac{q-1}{2}-m}$ is 
$$
\binom{n-1+\frac{q-1}{2}}{m} \sum_{i=1}^n c_i (-d_i)^{m}=
\binom{n-1+\frac{q-1}{2}}{m} \not \equiv 0 \pmod p,
$$
and we have $\operatorname{deg} f \leq n-1+\frac{q-1}{2}$. 

For each $0 \leq k \leq n-2$ and $1\leq j \leq N$, by Corollary \ref{cor}, we have
\begin{align*}
E^{(k)} f (a_j)
&= \binom{n-1+\frac{q-1}{2}}{k}  \sum_{i=1}^n c_i (a_j-d_i)^{n-1+\frac{q-1}{2}-k} \\
&= \binom{n-1+\frac{q-1}{2}}{k}  \sum_{i=1}^n c_i (a_j-d_i)^{n-1-k}   \\
&= \binom{n-1+\frac{q-1}{2}}{k}  \sum_{l=0}^{n-1-k} \binom{n-1-k}{l} \bigg(\sum_{i=1}^n c_i (-d_i)^l\bigg)a_j^{n-1-k-l}.
\end{align*}
If $\binom{n-1+\frac{q-1}{2}}{k} \not \equiv 0 \pmod p$, then for each $0 \leq l \leq n-1-k \leq n-1$ such that ${\binom{n-1-k}{l} \not \equiv 0 \pmod p}$, we have $l \in L(n)$ and $\sum_{i=1}^n c_i (-d_i)^l=0$. So we have ${E^{(k)} f (a_j)=0}$.

Note that $0 \in L(n)$, so for each $a_j \notin D$, we additionally have
\begin{align*}
E^{(n-1)} f (a_j)
= \binom{n-1+\frac{q-1}{2}}{n-1}  \sum_{i=1}^n c_i (a_j-d_i)^{\frac{q-1}{2}} = \binom{n-1+\frac{q-1}{2}}{n-1}  \sum_{i=1}^n d_i  =0.
\end{align*}
Now by Lemma \ref{root}, each $a_j \in D$ is a root of $f$ of multiplicity at least $n-1$, and each $a_j \notin D$ is a root of $f$ of multiplicity at least $n$. Therefore
$$(n-1)n +n (N-n)=n N-n \leq n-1+\frac{q-1}{2},$$ i.e. $n(N-2) \leq \frac{q-3}{2}$.
\end{proof}

\begin{rem}\rm
We do need the assumption that $m \leq \frac{q-1}{2}$, otherwise it is possible that $0<m'=m-\frac{q-1}{2} \in L(n)$, then $\sum_{i=1}^n c_i (-d_i)^{m'}=0$ will imply that $\sum_{i=1}^n c_i (-d_i)^m=0$ as each $d_i$ is a quadratic residue.
\end{rem}

Consider the $(|L(n)|+1) \times N$ matrix 
\begin{equation}\label{amn}
A_{m,n}: =\big((-a_i)^l\big)_{1 \leq i \leq N, l \in L(n) \cup \{m\}}.    
\end{equation}
Note that the coefficient matrix of the system \eqref{refined_system}  is a submatrix of $A_{m,n}$.

\begin{lem}\label{lee}
If $n \leq N$ and $n-1 \equiv \frac{p+1}{2} \pmod  p$, then $|L(n)|<n$.
\end{lem}

\begin{proof}
Since $n-1 \equiv \frac{p+1}{2} \pmod  p$, we have $p \mid n-1+\frac{q-1}{2}$. If $l \in L(n)$, then there exists $k$ such that $0 \leq k \leq n-1$ and $\binom{n-1+\frac{q-1}{2}}{k}\binom{n-1-k}{l} \not \equiv 0 \pmod p$. Now by Lucas's theorem, we must have $p \mid k$ and $n-1-k \equiv \frac{p+1}{2} \pmod p$. Since $\binom{n-1-k}{l} \not \equiv 0 \pmod p$, by Lucas's theorem, $l \equiv 0,1, \ldots \frac{p+1}{2} \pmod p$. Then in particular, $\frac{p+3}{2} \notin L(n)$, and $|L(n)|<n$.
\end{proof}

\begin{lem}\label{t6}
Suppose $n \leq N=\omega(P_q)$ be such that $n-1 \equiv \frac{p+1}{2} \pmod  p$,  and ${\frac{q-1}{2} \geq m \geq n}$ be such that $\binom{n-1+\frac{q-1}{2}}{m} \not \equiv 0 \pmod p$. Suppose further that for any $n$-subset $D$ of $C$, the above system of equations \eqref{refined_system} has no solution. Then the last row of $A_{m,n}$ is a linear combination of the first $|L(n)|$ rows, where $A_{m,n}$ is defined in equation \eqref{amn}.
\end{lem}

\begin{proof}
 Note that by Lemma \ref{lee}, $|L(n)|+1 \leq n \leq N$. If $A=A_{m,n}$ has full rank, which equals to $|L(n)|+1$, then $A$ has an invertible $(|L(n)|+1) \times (|L(n)|+1)$ sub-matrix, which columns correspond to a $(|L(n)|+1)$-subset $H$ of $C$. Then for any $n$-subset $D$ of $C$ containing $H$, the coefficient matrix of \eqref{refined_system} in Theorem \ref{t5} has full rank, and thus the system has a solution. So by our assumption, $A$ does not have full rank, which means the rows of $A$ are linearly dependent. Note that the first $|L(n)|$ rows of $A$, i.e., those rows with $l \in L(n)$, form a sub-matrix of the Vandermonde matrix $\big((-a_i)^j\big)_{1 \leq i \leq N, 0 \leq j \leq N-1}$, so the first $|L(n)|$ rows are linearly independent. Therefore, the last row of $A$ is a linear combination of the first $|L(n)|$ rows. 
\end{proof}

We focus on the case $s\geq 2$. In view of the proof of Theorem \ref{t3}, we see if $z_{s-1} < \frac{p-1}{2}$, we can get 
$N \leq \frac{1}{2}(\sqrt{2q-1}+1)+p^{s-1},$
which is a good upper bound. In the case $z_{s-1} = \frac{p-1}{2}$, we could instead let
$$
n-1=(z_s, z_{s-1}-1, \frac{p-1}{2}, \ldots, \frac{p-1}{2})_p,
$$
to get the improved upper bound
$N \leq \frac{1}{2}(\sqrt{2q-1}+1)+2p^{s-1}.$ Therefore, we see that the case $z_{s-1} \leq \frac{p-1}{2}$ is consistent with Conjecture \ref{conj1}. In the following discussion, we will focus on the case $z_{s-1}>\frac{p+1}{2}$. We assume   
$$
n-1=(z_s, z_{s-1}, \frac{p-1}{2}, \ldots, \frac{p-1}{2},\frac{p+1}{2})_p,
$$
is the largest number of this form no greater than $N-1$, 
then
$$
n-1+\frac{q-1}{2}=(\frac{p-1}{2}, \ldots, \frac{p-1}{2},z_s', z_{s-1}', 0, \ldots, 0)_p,
$$
where $z_s'= z_s+\frac{p+1}{2},z_{s-1}'=z_{s-1}-\frac{p-1}{2}$, and we have $n \leq N < n+p^{s-1}$. 

% We need to first understand the structure of the set $L(n)$ as well those integers $m$ with the required properties. Note that $z_{s-1}' \leq \frac{p-1}{2}$, then it is easy to verify that
%$$
%M=\{m=(c_{2s}, \ldots, c_s,c_{s-1},0, \ldots, 0)_p : m \geq n, c_j\leq \frac{p-1}{2} \mbox{  for } s\leq j \leq 2s, \mbox{ and } c_{s-1}\leq z_{s-1}'\}
%$$
Let $M$ denote the set of all possible integers $m$ such that 
$$
\binom{n-1+\frac{q-1}{2}}{m} \not \equiv 0 \pmod p, \text{ and } n \leq m \leq \frac{q-1}{2}.
$$
Using Lucas's theorem, it is easy to verify that
$$
M=\{m=(c_{2s}, \ldots, c_s,c_{s-1}, \ldots, 0)_p \geq n : c_j\leq \frac{p-1}{2} \mbox{  for } s\leq j \leq 2s, \mbox{ and } c_{s-1}\leq z_{s-1}'\}.
$$
We can also determine the structure of the set $L(n)$:
\begin{lem} \label{L}
If $l=(l_s, l_{s-1}, \ldots,l_0)_p \in L(n)$, then $0 \leq l_0 \leq \frac{p+1}{2}$ and $0 \leq l_j \leq \frac{p-1}{2}$ for each $1 \leq j \leq s-2$.
\end{lem}

\begin{proof}
If $l=(l_s, l_{s-1}, \ldots,l_0)_p \in L(n)$, then there exists $0 \leq k \leq n-1$ such that $\binom{n-1+\frac{q-1}{2}}{k} \not \equiv 0 \pmod p$ and $\binom{n-1-k}{l} \not \equiv 0 \pmod p$. Note that $p^{s-1} \mid n-1+\frac{q-1}{2}$, then by Lucas's theorem, we must have $p^{s-1} \mid k$. So 
$$n-1-k \equiv n-1 \equiv \bigg(\frac{p-1}{2}, \ldots, \frac{p-1}{2},\frac{p+1}{2}\bigg)_p \pmod {p^{s-1}}.$$ 
Since $\binom{n-1-k}{l} \not \equiv 0 \pmod p$, then we need 
$$
\binom{(\frac{p-1}{2}, \ldots, \frac{p-1}{2},\frac{p+1}{2})_p}{(l_{s-2}, \ldots,l_0)_p}
=\binom{\frac{p+1}{2}}{l_0} \prod_{j=1}^{s-2} \binom{\frac{p-1}{2}}{l_j} 
\not \equiv 0 \pmod p.
$$
Therefore, $0 \leq l_0 \leq \frac{p+1}{2}$ and $0 \leq l_j \leq \frac{p-1}{2}$ for each $1 \leq j \leq s-2$.
\end{proof}

Conjecture \ref{conj1} could be proved if we showed the existence of an $m \in M$ with the following properties.

\begin{conj} \label{con}
There is an integer $m \in M$ such that and the last row is linearly independent with the first $|L(n)|$ rows in the matrix $A_{m,n}$.
\end{conj}

We remark that Conjecture \ref{con} is closely related to the singularity of generalized Vandermonde matrices; a survey can be found in \cite[Section 4.5]{thesis}. The best known result, due to Shparlinski \cite{IS}, states that for a fixed integer $n$, as $q \to \infty$, almost all $n \times n$ generalized Vandermonde matrices over $\F_q$ are non-singular. However, it seems this result is not strong enough to prove Conjecture \ref{con}. We end the section by showing the following proposition.

\begin{prop}\label{t7}
Conjecture \ref{con} implies Conjecture \ref{conj1}.
\end{prop}

\begin{proof}
Let $m \in M$ be such that the last row is linearly independent with the first $|L(n)|$ rows in the matrix $A_{m,n}$. Then in view of the proof of Lemma \ref{t6}, there exists a $n$-subset $D$ of $C$ such that the system \eqref{refined_system} has a solution. Since $n \leq m \leq \frac{q-1}{2},\binom{n-1+\frac{q-1}{2}}{m} \not \equiv 0 \pmod p$, by Theorem \ref{t5}, we have $n(N-2) \leq \frac{q-3}{2}$. By the construction of $n$, we have $n \leq N < n+p^{s-1}$. Then we get $(N-p^{s-1})(N-2) \leq \frac{q-3}{2}$,  and therefore
\begin{align*}
N 
&\leq \frac{p^{s-1}+2}{2}+\sqrt{\frac{q-3}{2}-2p^{s-1}+\bigg(\frac{p^{s-1}+2}{2}\bigg)^2}\\
&\leq \frac{p^{s-1}+2}{2}+\sqrt{\frac{q-1}{2}+\frac{p^{2s-2}}{4}}\\
&\leq \frac{p^{s-1}+2}{2}+\sqrt{\frac{q}{2}}+\frac{p^{2s-2}}{8\sqrt{\frac{q}{2}}}\\
%&= \sqrt{\frac{q}{2}}+ \frac{p^{s-1}}{2}+\frac{\sqrt{2}}{8} p^{s-5/2}+1    \\
%&= \sqrt{\frac{q}{2}}+ p^{s-1}\bigg(\frac{1}{2}+\frac{\sqrt{2}}{8} p^{-3/2}+p^{1-s}\bigg)\\
&<\sqrt{\frac{q}{2}}+\frac{12+\sqrt{2}}{8}p^{s-1}.
\end{align*}
\end{proof}

\section*{Acknowledgement}
The author would like to thank J\'ozsef Solymosi and Ethan White for their valuable suggestions, and Greg Martin and Joshua Zahl for helpful discussions. The author also would like to thank the anonymous referees for a careful reading of the draft.

\end{document}